\documentclass[a4paper,11pt,twoside]{article}
\usepackage{booktabs}
\usepackage{amsmath,amssymb,amsfonts,amsthm,stmaryrd}
\usepackage{indentfirst}
\usepackage{color}
\usepackage{geometry}
\usepackage{tikz}
\usepackage{graphicx}
\usepackage{wrapfig}
\usepackage{fancyhdr}
\usepackage{hyperref}
\hypersetup{
    colorlinks=true,
    citecolor=blue,
    linkcolor=red,
    filecolor=magenta,
    urlcolor=cyan,
}

\newtheorem{theorem}{Theorem}

\newtheorem{corollary}[theorem]{Corollary}
\newtheorem{lemma}[theorem]{Lemma}

\newtheorem{definition}[theorem]{Definition}
\newtheorem{example}[theorem]{Example}

\newtheorem{proposition}[theorem]{Proposition}

\newcommand{\overbar}[1]
           {\mkern 1.5mu\overline{\mkern-1.5mu#1\mkern-1.5mu}\mkern 1.5mu}

\def\e{\mathrm{e}}

\def\Li{\mathrm{Li}}
\let\set\mathbb
\let\cal\mathcal
\def\lc{\operatorname{lc}}

\DeclareMathOperator{\quot}{Quot}

\DeclareMathOperator{\re}{Re}
\DeclareMathOperator{\im}{Im}

\title{D-finite Numbers\footnote{DOI: 10.1142/S1793042118501099}}

\author{\medskip Hui Huang\footnote{Most of the work presented in this
    article was carried out while Huang was affiliated at the Johannes
    Kepler University Linz and supported by the Austrian Science Fund
    (FWF) grant W1214-N15 (project part~13).} \\
    \smallskip
    {\footnotesize\em David R. Cheriton School of Computer Science, University
      of Waterloo}\\[-1ex]
    {\footnotesize\em 200 University Avenue West, Waterloo, Ontario,
      N2L 3G1,Canada}\\[-.5ex]
    {\footnotesize\em hui.huang@uwaterloo.ca}\\
    \medskip\\
    Manuel Kauers\footnote{Supported by the Austrian
      Science Fund (FWF) grants Y464-N18 and F5004.}\\
    {\footnotesize\em Institute for Algebra, Johannes Kepler
      University}\\[-.5ex]
    {\footnotesize\em Altenberger Strasse 69,
      4040 Linz, Austria}\\[-.5ex]
    {\footnotesize\em manuel.kauers@jku.at} }
\date{}
\begin{document}
\maketitle

\begin{abstract}
  \noindent
  D-finite functions and P-recursive sequences are defined in terms of
  linear differential and recurrence equations with polynomial
  coefficients.  In this paper, we introduce a class of numbers
  closely related to D-finite functions and P-recursive sequences.  It
  consists of the limits of convergent P-recursive sequences.
  Typically, this class contains many well-known mathematical
  constants in addition to the algebraic numbers.  Our definition of
  the class of D-finite numbers depends on two subrings of the field
  of complex numbers.  We investigate how different choices of these
  two subrings affect the class.  Moreover, we show that D-finite
  numbers are essentially limits of D-finite functions at the point
  one, and evaluating D-finite functions at non-singular algebraic
  points typically yields D-finite numbers.  This result makes it
  easier to recognize certain numbers to be D-finite.

  \medskip\noindent
  {\em Keywords:} Complex numbers; D-finite functions; P-recursive sequences;
  algebraic numbers; evaluation of special functions.

  \medskip\noindent
  {Mathematics Subject Classification 2010:} 11Y35, 33E30, 33F05, 33F10
\end{abstract}

\section{Introduction}\label{SEC:intro}

D-finite functions have been recognized long
ago~\cite{Stan1980,Lips1989,Zeil1990b,SaZi1994,Mall1996,Stan1999} as
an especially attractive class of functions.  They are interesting on
the one hand because each of them can be easily described by a finite
amount of data, and efficient algorithms are available to do exact as
well as approximate computations with them.  On the other hand, the
class is interesting because it covers a lot of special functions
which naturally appear in various different context, both within
mathematics as well as in applications.

The defining property of a {\em D-finite} function is that it
satisfies a linear differential equation with polynomial
coefficients. This differential equation, together with an appropriate
number of initial terms, uniquely determines the function at
hand. Similarly, a sequence is called {\em P-recursive} (or rarely,
{\em D-finite}) if it satisfies a linear recurrence equation with
polynomial coefficients.  Also in this case, the equation together
with an appropriate number of initial terms uniquely determines the
object.

In a sense, the theory of D-finite functions generalizes the theory of
algebraic functions. Many concepts that have first been introduced for
the latter have later been formulated also for the former. In
particular, every algebraic function is D-finite (Abel's theorem), and
many properties the class of algebraic function enjoys carry over to
the class of D-finite functions.

The theory of algebraic functions in turn may be considered as a
generalization of the classical and well-understood class of algebraic
numbers. The class of algebraic numbers suffers from being relatively
small. There are many important numbers, most prominently the numbers
$\e$ and~$\pi$, which are not algebraic.

Many larger classes of numbers have been proposed, let us just mention
three examples.  The first is the class of periods (in the
sense of Kontsevich and Zagier~\cite{KoZa2001}).  These numbers are
defined as the values of multivariate definite integrals of algebraic
functions over a semi-algebraic set.  In addition to all the algebraic
numbers, this class contains important numbers such as~$\pi$, all zeta
constants (the Riemann zeta function evaluated at an integer)
and multiple zeta values, but it is so far
not known whether for example~$\e$, $1/\pi$ or Euler's
constant~$\gamma$ are periods (conjecturally they are not).  
The second example is the class of all numbers that appear as values of
so-called G-functions (in the sense of Siegel~\cite{Sieg2014}) at
algebraic number arguments~\cite{FiRi2014,FiRi2016a}.  The class of
G-functions is a subclass of the class of D-finite
functions, and it inherits some useful properties of that class.
Among the values that G-functions can assume are $\pi$, $1/\pi$,
values of elliptic integrals, and multiple zeta values, but it is so
far not known whether for example~$\e$, Euler's constant $\gamma$ or a
Liouville number are such a value (conjecturally they are not).

Another class of numbers is the class of holonomic constants, studied
by Flajolet and Vall{\' e}e~\cite[\S 4]{FlVa2000}.  (We thank Marc
Mezzarobba for pointing us to this reference.)  A constant that is the
value $f(z_0)$ of a D-finite function $f(z)$ at an algebraic point
$z_0$ where $f(z)$ is regular (i.e., analytic) is called a {\em
  regular holonomic constant}. Classical examples are $\pi$, $\log(2)$
and the polylogarithm value $\Li_4(1/2)$. A {\em singular holonomic
  constant} is defined to be the value of a D-finite function $f(z)$
at a Fuchsian singularity (also known as regular
singularity~\cite{Waso1987}) of a defining differential equation
for~$f(z)$. Note that the classes of regular and singular holonomic
constants are not completely opposite to each other, since a constant
can be of both types. A typical example is Ap{\' e}ry's
constant~$\zeta(3)$. This constant is of singular type since
$\zeta(3)=\Li_3(1)$ where the polylogarithm function $\Li_3(z)$ is
D-finite and has a singularity at one of the Fuchsian type.  On the
other hand, $\zeta(3)$ is also a regular holonomic constant, because
$\Li_3(z)$ is a G-function and values of G-functions at algebraic
numbers are all of regular type by \cite[Theorem~1]{FiRi2014}.

It is tempting to believe that there is a strong relation between
holonomic constants and limits of convergent P-recursive sequences.
To make this relation precise, we introduce the class of D-finite
numbers in this paper. 

\begin{definition}\label{DEF:dfinitenos}
  Let $R$ be a subring of~$\set C$ and let $\set F$ be a subfield
  of~$\set C$.
  \begin{enumerate}
  \item
    A number $\xi\in\set C$ is called \emph{D-finite} (with respect to
    $R$ and~$\set F$) if there exists a convergent sequence
    $(a_n)_{n=0}^\infty$ in $R^{\set N}$ with $\lim_{n\to\infty}
    a_n=\xi$ and some polynomials $p_0,\dots,p_r\in\set F[n]$,
    $p_r\neq0$, such that
    \[
    p_0(n)a_n + p_1(n)a_{n+1} + \cdots + p_r(n)a_{n+r} = 0
    \]
    for all $n\in\set N$.

  \medskip
  \item
    The set of all D-finite numbers with respect to $R$ and $\set F$
    is denoted by $\cal D_{R,\set F}$. If $R=\set F$, we also write
    $\cal D_{\set F}:=\cal D_{\set F,\set F}$ for short.
  \end{enumerate}
\end{definition}

It is clear that $\cal D_{R,\set F}$ contains all the elements of~$R$,
but it typically contains many further elements.  For example, let $i$
be the imaginary unit, then $\cal D_{\set Q(i)}$ contains many (if not
all) the periods and, as we will see below, all the values of
G-functions as well as many (if not all) regular holonomic constants.
In addition, it is not hard to see that $\e$ and $1/\pi$ are D-finite
numbers.  According to Fischler and Rivoal's work~\cite{FiRi2016a},
also Euler's constant~$\gamma$ and any value of the Gamma function at
a rational number are D-finite.  (We thank Alin Bostan for pointing us
to this reference.)  We will show below that D-finite numbers are
essentially the limiting values of D-finite functions
at one. Moreover, the values D-finite functions can assume at
non-singular algebraic points are in fact D-finite numbers.  Together
with the work on arbitrary-precision evaluation of D-finite
functions~\cite{ChCh1990,vdHo1999,vdHo2001,vdHo2007a,Mezz2010,MeSa2010},
it follows that D-finite numbers are computable in the sense that for
every D-finite number~$\xi$ there exists an algorithm which for any
given $n\in\set N$ computes a numeric approximation of $\xi$ with a
guaranteed precision of $10^{-n}$. Consequently, all non-computable
numbers have no chance to be D-finite.  Besides these artificial
examples, we do not know of any explicit real numbers which are not
in~$\cal D_{\set Q}$, and we believe that it may be very difficult to
find some.

The definition of D-finite numbers given above involves two subrings
of $\set C$ as parameters: the ring to which the sequence terms of the
convergent sequences are supposed to belong, and the field to which
the coefficients of the polynomials in the recurrence equations should
belong.  Obviously, these choices matter, because we have, for
example, $\cal D_{\set R}=\set R\neq\set C=\cal D_{\set C}$.  Also,
since $\cal D_{\set Q}$ is a countable set, we have $\cal D_{\set
  Q}\neq\cal D_{\set R}$.  On the other hand, different choices of $R$
and $\set F$ may lead to the same classes.  For example, we would not
get more numbers by allowing $\set F$ to be a subring of~$\set C$
rather than a field, because we can always clear denominators in a
defining recurrence.  One of the goals of this article is to
investigate how $R$ and $\set F$ can be modified without changing the
resulting class of D-finite numbers.

As a long-term goal, we hope to establish the notion of D-finite
numbers as a class that naturally relates to the class of D-finite
functions in the same way as the classical class of algebraic numbers
relates to the class of algebraic functions.

\section{D-finite Functions and P-recursive Sequences}\label{SEC:dfinitefun}

\def\<#1>{\langle#1\rangle}
Throughout the paper, $R$ is a subring of $\set C$ and $\set F$ is a
subfield of~$\set C$, as in Definition~\ref{DEF:dfinitenos} above.  We
consider linear operators that act on sequences or power series and
analytic functions.  We write $S_n$ for the shift operator w.r.t.~$n$
which maps a sequence $(a_n)_{n=0}^\infty$ to
$(a_{n+1})_{n=0}^\infty$.  The set of all linear operators of the form
$L:=p_0+p_1S_n+\cdots+p_rS_n^r$, with $p_0,\dots,p_r\in\set F[n]$,
forms an Ore algebra; we denote it by $\set F[n]\<S_n>$.  Analogously,
we write $D_z$ for the derivation operator w.r.t.~$z$ which maps a
power series or function $f(z)$ to its derivative
$f'(z)=\frac{d}{dz}f(z)$.  Also the set of linear operators of the
form $L:=p_0+p_1D_z+\cdots+p_rD_z^r$, with $p_0,\dots,p_r\in\set
F[z]$, forms an Ore algebra; we denote it by $\set F[z]\<D_z>$.  For
an introduction to Ore algebras and their actions,
see~\cite{BrPe1996}.  When $p_r\neq0$, we call $r$ the \emph{order} of
the operator and $\lc(L):=p_r$ its \emph{leading coefficient.}

\begin{definition}\label{DEF:dfinite}\leavevmode\null
  \begin{enumerate}
  \item
    A sequence $(a_n)_{n=0}^\infty\in R^{\set N}$ is called {\em
      P-recursive} or \emph{D-finite} over $\set F$ if there exists a
    nonzero operator $L= \sum_{j=0}^r p_j(n) S_n^j \in \set
    F[n]\langle S_n \rangle$ such that
    \[
    L\cdot a_n = p_r(n) a_{n+r} + \dots + p_1(n) a_{n+1} + p_0(n) a_n = 0
    \]
    for all $n\in \set N$.

    \medskip
  \item
    A formal power series $f(z)\in R[[z]]$ is called {\em D-finite}
    over $\set F$ if there exists a nonzero operator $L= \sum_{j=0}^r
    p_j(z) D_z^j \in \set F[z]\langle D_z\rangle$ such that
    \[
    L \cdot f(z) = p_r(z) D_z^r f(z) + \dots + p_1(z) D_z f(z) + p_0(z) f(z) = 0.
    \]

    \smallskip
  \item
    An analytic function $f\colon U\to\set C$ defined in some open set
    $U\subseteq\set C$ is called \emph{D-finite} over $\set F$ if
    there exists a nonzero operator $L= \sum_{j=0}^r p_j(z) D_z^j \in
    \set F[z]\langle D_z\rangle$ such that
    \[
    L \cdot f(z) = p_r(z) D_z^r f(z) + \dots + p_1(z) D_z f(z) + p_0(z) f(z) = 0
    \]
    for all $z\in U$.

    \medskip
  \item
    A formal power series $f(z) \in \set F[[z]]$ is called {\em
      algebraic} over $\set F$ if there exists a nonzero bivariate
    polynomial $P(z,y) \in \set F[z,y]$ such that $P(z,f(z)) = 0$.

    \medskip
  \item
    An analytic function $f\colon U\to\set C$ defined in some open set
    $U\subseteq\set C$ is called \emph{algebraic} over~$\set F$ if
    there exists a nonzero bivariate polynomial $P(z,y)\in\set F[z,y]$
    such that $P(z,f(z))=0$ for all $z\in U$.
  \end{enumerate}
\end{definition}

Unless there is a danger of confusion, we will not strictly
distinguish between complex functions that are analytic in some
neighborhood of zero and the formal power series appearing as their
Taylor expansions at zero. In particular, if the formal power series
$f\in\set F[[z]]$ happens to be convergent, we may as well regard it
as an analytic function defined in some open neighborhood of zero.

A formal power series (or function) is D-finite if and only if its
coefficient sequence is P-recursive.  Many functions like
exponentials, logarithms, sine, arcsine and hypergeometric series, as
well as many formal power series like $\sum n! z^n$, are
D-finite. Hence their respective coefficient sequences are
P-recursive.

The class of D-finite functions (resp.\ P-recursive sequences) is
closed under certain operations: addition, multiplication, derivative
(resp.\ forward shift) and integration (resp.\ summation). In
particular, the set of D-finite functions (resp.\ P-recursive
sequences) forms a left-$\set F[z]\<D_z>$-module (resp.\ a left-$\set
F[n]\<S_n>$-module). Also, if $f$ is a D-finite function and $g$ is an
algebraic function, then the composition $f\circ g$ is D-finite. These
and further closure properties are easily proved by linear algebra
arguments, proofs can be found for instance
in~\cite{Stan1980,SaZi1994,KaPa2011}. We will make free use of these
facts.

We will be considering singularities of D-finite functions. Recall
from the classical theory of linear differential
equations~\cite{Ince1944} that a linear differential equation
$p_0(z)f(z)+\cdots+p_r(z)f^{(r)}(z)=0$ with polynomial coefficients
$p_0,\dots,p_r\in\set F[z]$ and $p_r\neq0$ has a basis of analytic
solutions in a neighborhood of every point $\xi\in\set C$, except
possibly at roots of~$p_r$. The roots of $p_r$ are therefore called
the \emph{singularities} of the equation (or the corresponding linear
operator).  All other points are {\em ordinary points} (or {\em
  non-singular points}) of the equation.  The behaviour of a D-finite
function near a singularity $z_0$ can in general not be described by a
formal power series, but it is always a linear combination of
generalized series of the form
\[
\exp(P((z-z_0)^{-1/s}))(z-z_0)^\alpha a\bigl((z-z_0)^{1/s},
\log(z-z_0)\bigr)
\]
for some $s\in\set N$, $P\in\bar{\set F}[z]$, $\alpha\in\bar{\set F}$,
and $a\in\bar{\set F}[[x]][y]$.  See~\cite{Ince1944} for details of
this construction. Formal power series are in general also not
sufficient to describe the behaviour for algebraic functions, but such
functions are always linear combinations of so-called Puiseux series,
which can be written in the form
\[
\sum_{n=n_0}^\infty a_n (z-z_0)^{n/s}
\]
for some $n_0\in\set Z$ and some positive integer $s$. See,
e.g.,~\cite{Walk1950} for details.

It can happen that $\xi\in\set C$ is a singularity of the equation but
the equation nevertheless admits a basis of analytic solutions at this
point. Such a singularity is called an \emph{apparent singularity.}
It is well-known~\cite{Ince1944,CKS2016} that for any given linear
differential equations with some apparent and some non-apparent
singularities, we can always construct another linear differential
equation (typically of higher order) whose solution space contains the
solution space of the first equation and whose only singularities are
the non-apparent singularities of the first equation. This process is
known as desingularization. For later use, we will give a proof of the
composition closure property for D-finite functions which pays
attention to the singularities.

\begin{theorem}\label{THM:algsubs}
  Let $P(z,y)\in\set F[z,y]$ be a polynomial of degree~$d$ in~$y$, and
  let $L\in\set F[z]\<D_z>$ nonzero.  Let~$\zeta\in\set C$ be such
  that $P$ defines $d$ distinct analytic algebraic functions $g(z)$
  with $P(z,g(z))=0$ in a neighborhood $\Omega$ of~$\zeta$, and assume
  that for none of these functions, the value $g(\zeta)\in\set C$ is a
  singularity of~$L$.  Fix a solution~$g$ of~$P$ and an analytic
  solution $f$ of~$L$ defined in a neighborhood of~$g(\zeta)$. Then
  there exists a nonzero operator $M\in\set F[z]\langle D_z\rangle$
  with $M\cdot(f\circ g)=0$ which does not have $\zeta$ among its
  singularities.  Moreover, any point in the neighborhood $\Omega$
  with the property that none of the evaluations at this point of the
  $d$ solutions of $P$ near $\zeta$ gives a singularity of $L$, is an
  ordinary point of $M$.
\end{theorem}
\begin{proof} (borrowed from \cite{KaPo2017})
  Let $g$ be a root of $P$ near~$\zeta$.  If $g$ is constant, then so
  is $f\circ g$ and we can take $M=D_z$. Suppose that $g$ is not
  constant.  Without loss of generality, we may assume that $P$ is
  irreducible (if it is not, replace $P$ by the minimal polynomial
  of~$g$). Then none of the solutions of $P$ is constant.
	
  Consider the operator $\tilde L:=L(g,(g')^{-1}D_z)\in
  \overline{\set F(z)}\<D_z>$.
  Because of $D_z\cdot(f\circ g)=(f'\circ g)g'$, we
  have $L\cdot f=0$ if and only if $\tilde L\cdot(f\circ g)=0$.
  Therefore, if $f_1,\dots,f_r$ is a basis of the solution space of
  $L$ near $g(\zeta)$, then $f_1\circ g,\dots,f_r\circ g$ is a basis
  of the solution space of $\tilde L$ near~$\zeta$.
	
  Let $g_1,\dots,g_d$ be all the solutions of $P$ near~$\zeta$, and
  let $M$ be the least common left multiple of all the operators
  $L(g_j,(g_j')^{-1}D_z)$.  Then the solution space of $M$ near~$\zeta$
  is generated by all the functions $f_i\circ g_j$.  The
  Galois group $G$ of $P$ consists of all automorphisms of the field
  $K=\set F(z,g_1,\dots,g_d)$ which leave $\set F(z)$ fixed.  The
  Galois group respects the differential structure of the field~$K$ in
  the sense that for all $u\in K$ and all $\pi\in G$ we have
  $\pi(u')=\pi(u)'$.  Therefore, the action of $G$ on $K$ naturally
  extends to an action of $G$ on the ring $K\<D_z>$ of linear
  differential operators.  Since $M$ is the least common left multiple
  of all the operators $L(g_j,(g_j')^{-1}D_z)$, regardless of their
  order, we have $\pi(M)=M$ for all $\pi\in G$.  This implies that
  $M\in\set F(z)\<D_z>$.  (This argument already appears in Section~61
  of~\cite{schl1895}.)
	
  After clearing denominators (from the left) if necessary, we may
  assume that $M$ is an operator in $\set F[z]\<D_z>$ whose solution
  space is generated by functions that are analytic at~$\zeta$. Since
  $g_j$ are analytic at any point $\eta$ in the neighborhood $\Omega$,
  the functions that generate the solution space of~$M$ are also
  analytic at $\eta\in \Omega$ provided that none of the values
  $g_j(\eta)$ is a singularity of~$L$.  Therefore, by the remarks made
  about desingularization, it is possible to replace $M$ by an
  operator (possibly of higher order) which does not have $\zeta$ and
  such $\eta$ among its singularities.
\end{proof}

By a similar argument, we see that algebraic extensions of the
coefficient field of the recurrence operators are useless. Moreover,
it is also not useful to make $\set F$ bigger than the quotient field
of~$R$.

\goodbreak
\begin{lemma}\label{lemma:changeopseq}\leavevmode\null
  \begin{enumerate}
  \item\label{lemma:changeopseq1}
    If $\set E$ is an algebraic extension field of $\set F$ and
    $(a_n)_{n=0}^\infty$ is P-recursive over $\set E$, then it is also
    P-recursive over~$\set F$.

    \medskip
  \item\label{lemma:changeopseq2}
    If $R\subseteq\set F$ and $(a_n)_{n=0}^\infty\in R^{\set N}$ is
    P-recursive over~$\set F$, then it is also P-recursive over
    $\quot(R)$, the quotient field of~$R$.

    \medskip
  \item\label{lemma:changeopseq3}
    If $\set F$ is closed under complex conjugation and
    $(a_n)_{n=0}^\infty$ is P-recursive over~$\set F$, then so are
    $(\bar a_n)_{n=0}^\infty$, $(\re(a_n))_{n=0}^\infty$, and
    $(\im(a_n))_{n=0}^\infty$.
  \end{enumerate}
\end{lemma}
\begin{proof}
  \begin{enumerate}
  \item Let $L\in\set E[n]\<S_n>$ be an annihilating operator of
    $(a_n)_{n=0}^\infty$. Then, since $L$ has only finitely many
    coefficients, $L\in\set F(\theta)[n]\<S_n>$ for some
    $\theta\in\set E$. Let $M$ be the least common left multiple of
    all the conjugates of~$L$.  Then $M$ is an annihilating operator
    of~$(a_n)_{n=0}^\infty$ which belongs to $\set F[n]\<S_n>$. The
    claim follows.

    \medskip
  \item
    Let us write $\set K=\quot(R)$.  Let $L\in\set F[n]\<S_n>$ be a
    nonzero annihilating operator of~$(a_n)_{n=0}^\infty$.  Since
    $\set F$ is an extension field of~$\set K$, it is a vector space
    over~$\set K$. Write
    \[
    L = \sum_{m=0}^r \sum_{j=0}^{d_m} p_{mj} n^j S_n^m,
    \]
    where $r, d_m \in \set N$ and $p_{mj} \in \set F$ not all
    zero. Then the set of the coefficients $p_{ij}$ belongs to a
    finite dimensional subspace of~$\set F$. Let $\{\alpha_1, \dots,
    \alpha_s\}$ be a basis of this subspace over~$\set K$. Then for
    each pair $(m,j)$, there exists $c_{mj\ell} \in \set K$ such that
    $p_{mj} = \sum_{\ell = 1}^s c_{mj\ell} \alpha_\ell$, which gives
    \[
    0 = L\cdot a_n = \sum_{\ell=1}^s \alpha_\ell
    \underbrace{\left(\sum_{m=0}^r\sum_{j=0}^{d_m} c_{mj\ell} n^j
      a_{n+m}\right)}_{=:b_n\in\set K}.
    \]
    For all $n\in \set N$, it follows from the linear independence of
    $\{\alpha_1, \dots, \alpha_s\}$ that $b_n=0$.  Therefore
    \[
    \sum_{m=0}^r\underbrace{\left(\sum_{j=0}^{d_m} c_{mj\ell}
      n^j\right)}_{\in \set K[n]}S_n^m\cdot a_n = 0 \quad \text{for
      all}\ n\in \set N \ \text{and}\ \ell = 1, \dots, s.
    \]
    Thus $(a_n)_{n=0}^\infty$ has a nonzero annihilating operator with
    coefficients in $\set K[n]$.

    \medskip
  \item Since $(a_n)_{n=0}^\infty$ is P-recursive over~$\set F$, there
    exists a nonzero operator~$L$ in $\set F[n]\langle S_n\rangle$
    such that $L \cdot a_n = 0$.  Hence $\bar L \cdot {\bar a_n} =0$
    where $\bar L$ is the operator obtained from $L$ by taking the
    complex conjugate of each coefficient.  Since $\set F$ is closed
    under complex conjugation by assumption, $\bar L$~belongs to $\set
    F[n]\<S_n>$, and hence $(\bar a_n)_{n=0}^\infty$ is P-recursive
    over~$\set F$. Because of
    \[\re(a_n) = \frac12(a_n + \bar a_n)\quad \text{and}\quad
    \im(a_n) = \frac1{2i}(a_n - \bar a_n),\]
    where $i$ is the imaginary unit,
    the assertions follow by closure properties.
  \end{enumerate}
\end{proof}

\goodbreak
Of course, all the statements hold analogously for D-finite formal
power series instead of P-recursive sequences.

If a D-finite function is analytic in a neighborhood of zero, then it
can be extended by analytic continuation to any point in the complex
plane except for finitely many ones, namely the singularities of the
given function. Those closest to the origin are called {\em dominant
  singularities} of the function. In this sense, D-finite functions
can be evaluated at any non-singular point by means of analytic
continuation. Numerical evaluation algorithms for D-finite functions
have been developed
in~\cite{ChCh1990,vdHo1999,vdHo2001,vdHo2007a,Mezz2010,MeSa2010},
where the last two references also provide a {\sc Maple}
implementation, namely the {\sf NumGfun} package, for computing such
evaluations.  These algorithms perform arbitrary-precision evaluations
with full error control.

\section{Algebraic Numbers}\label{SEC:algnos}

Before turning to general D-finite numbers, let us consider the
subclass of algebraic functions. We will show that in this case, the
possible limits are precisely the algebraic numbers. For the purpose
of this article, let us say that a sequence $(a_n)_{n=0}^\infty\in\set
F^{\set N}$ is \emph{algebraic} over $\set F$ if the corresponding
power series~$\sum_{n=0}^\infty a_nz^n\in\set F[[z]]$ is algebraic in
the sense of Definition~\ref{DEF:dfinite}. Since algebraic functions
are D-finite, it is clear that algebraic sequences are P-recursive.
We will write $\cal A_{\set F}$ for the set of all complex numbers
which are limits of convergent algebraic sequences over~$\set F$.

Recall that two sequences $(a_n)_{n=0}^\infty$, $(b_n)_{n=0}^\infty$
with at most finitely many zero terms are called \emph{asymptotically
  equivalent,} written $a_n\sim b_n$ ($n\to\infty$), if the quotient
$a_n/b_n$ converges to one as $n$ tends to infinity. Similarly, two
complex functions $f(z)$ and $g(z)$ are called {\em asymptotically
  equivalent} at a point~$\zeta\in\set C$, written $f(z)\sim g(z)$
($z\to\zeta$), if the quotient $f(z)/g(z)$ converges to one as $z$
approaches~$\zeta$. These notions are connected by the following
classical theorem.

\begin{theorem}\label{THM:asympt}\leavevmode\null
  \begin{enumerate}
  \item\label{THM:transfer} (Transfer theorem~\cite{FlOd1990,FlSe2009})
    Assume that $f(z)\in \set F[[z]]$ is analytic at zero with the
    only dominant singularity $z=1$ and
    \[f(z) \sim \frac{1}{(1-z)^\alpha}\quad (z\to 1)\]
    with $\alpha \notin \{0,-1,-2,\dots\}$. Then
    \[
      [z^n]f(z)\sim \frac{n^{\alpha-1}}{\Gamma(\alpha)} \quad (n\to\infty),
     \]
     where $\Gamma(z)$ stands for the Gamma function and the notation
     $[z^n]f(z)$ refers to the coefficient of $z^n$ in $f(z)$.

     \medskip
   \item\label{THM:abelian} (Basic Abelian theorem~\cite{FGS2004})
     Let $(a_n)_{n=0}^\infty\in \set F^{\set N}$ be a sequence that
     satisfies the asymptotic estimate
     \[
     a_n\sim n^{\alpha}\quad (n\rightarrow\infty),
     \]
     where $\alpha \geq 0$.  Then the generating function $f(z) =
     \sum_{n=0}^\infty a_n z^n$ satisfies the asymptotic estimate
     \[
     f(z) \sim \frac{\Gamma(\alpha + 1)}{(1-z)^{\alpha+1}} \quad
     (z\rightarrow 1^{-}).
     \]
     This estimate remains valid when $z$ tends to one in any sector
     with vertex at one symmetric about the horizontal axis, and with
     opening angle less than~$\pi$.
  \end{enumerate}
\end{theorem}
Also recall the formal version of the implicit function
theorem~\cite{Soka2009}, which says that for any bivariate polynomial
$P(z,y)\in \set F[z,y]$ with $P(0,0) = 0$ and $(D_yP)(0,0) \neq 0$,
there exists a unique formal power series $f(z) \in \set F[[z]]$
with $f(0) =0$ so that $P(z,f(z)) = 0$.

With the above preparations, we are ready to develop the following key
lemma, which indicates that depending on whether $\set F$ is a real
field or not, every real algebraic number or every algebraic number
can appear as a limit of an algebraic sequence over $\set F$.

\begin{lemma}\label{LEM:minpoly}\leavevmode\null
  Let $p(y)\in\set F[y]$ be an irreducible polynomial of degree~$d$.
  Assume that $\zeta_1,\dots, \zeta_d\in \bar{\set F}$ are all the
  roots of $p$.  Then there exists a polynomial $P(z,y)\in\bar{\set F}[z,y]$
  of degree $d$ in~$y$ admitting $d$ distinct roots in
  $\bar{\set F}[[z]]$ such that for each $j = 1, \dots, d$, there is
  exactly one $f_j(z)\in\bar{\set F}[[z]]$ with $P(z,f_j(z))=0$ and
  $\lim_{n\to\infty}[z^n]f_j(z)=\zeta_j$. All these $f_j(z)$ are
  analytic at zero with the only possible dominant singularity at
  $z=1$, which can at most be a simple pole.  Furthermore, if
  \begin{align*}
    \text{either}\quad (\set F\subseteq\set R\ \text{and}\
    \zeta_1\in\bar{\set F}\cap\set R) \quad
    \text{or}\quad (\set F\setminus\set R\neq\emptyset),
  \end{align*}
  then $P(z,y)$ can be chosen in $\set F[z,y]$ so that $f_1(z)\in{\set F}[[z]]$.
\end{lemma}
\begin{proof}
  If $d =1$, then $p(y)= y-\zeta_1$ with $\zeta_1\in \set F$.  Letting
  $P(z,y) = p(y)$ yields the assertions.
	
  Now assume that $d>1$. Then $\zeta_j\neq 0$ for all $j=1, \dots, d$
  since $p$ is irreducible.  Let $\varepsilon>0$ be such that any two
  (real or complex) roots of $p$ have a distance of more
  than~$2\varepsilon$ to each other.  Such an $\varepsilon$ exists
  because $p$ is an irreducible polynomial of degree greater than one,
  and thus has only finitely many distinct roots.  The roots of a
  polynomial depend continuously on its coefficients.  Therefore there
  exists a real number~$\delta>0$ so that perturbing the coefficients
  by up to $\delta$ won't perturb the roots by more
  than~$\varepsilon/2$.  Any positive smaller number than $\delta$
  will have the same property.  By the choice of~$\varepsilon$, any
  such perturbation of the polynomial will have exactly one root in
  each of the open balls of radius $\varepsilon/2$ centered at the
  roots of~$p$.
	
  For fixed nonzero $\alpha\in \bar{\set F}$ with $|\alpha|<\delta/2$,
  consider the perturbation $\tilde P_\alpha(z,y)=p(y)-\alpha (1-z)\in
  \set F(\alpha)[z,y]$.  We will show that
  \begin{quote}
    \leavevmode\llap{$(\ast)$\quad}%
    the polynomial $\tilde P_\alpha(z,y)$ has exactly $d$ distinct
    roots in $\overline{\set F(z)}$ for fixed $z$ with $|z|\leq 1$,
    and any two of them have a distance of more than~$\varepsilon$.
    Moreover, there exist functions $g_1,\dots,g_d$ defined for
    $|z|\leq1$ such that $g_j(1) = \zeta_j$ and $\tilde
    P_\alpha(z,g_j(z))=0$ and $|g_j(z)-\zeta_j|<\varepsilon/2$ for all
    $z$ with $|z|\leq 1$ and $j=1,\dots,d$.
  \end{quote}
  In fact, since $|\alpha|<\delta/2$, for any~$z$ in the disk $|z|\leq
  1$ we have
  \[|{-}\alpha(1-z)|\leq 2|\alpha|<\delta.\]
  Therefore, for every $z$ with $|z|\leq 1$, each root of $\tilde
  P_\alpha(z,y)$ belongs to exactly one open ball of radius
  $\varepsilon/2$ centered at a root $\zeta_j$ of~$p(y)$, and by
  continuity, as $z$ varies and tends to one in the disk $|z|\leq 1$,
  each root approaches the root~$\zeta_j$ inside the corresponding
  open ball (Fig.~\ref{FIG:balls}).
  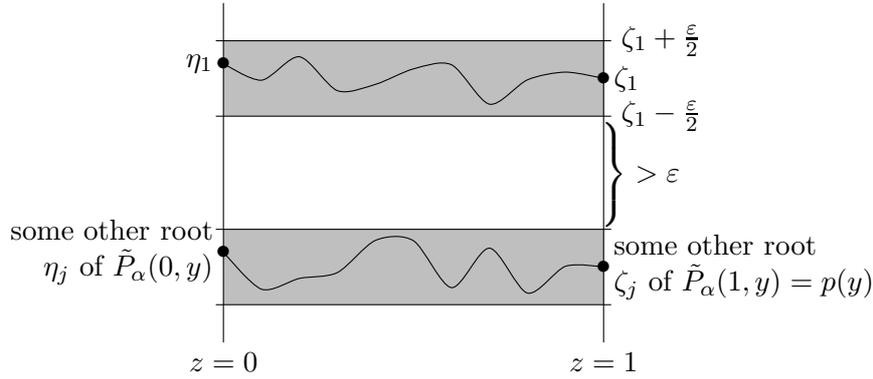
\begin{figure}[ht]
    \centering
    \begin{tikzpicture}
      \fill[lightgray] (0,.5) rectangle (5,1.5);
      \fill[lightgray] (0,3) rectangle (5,4);
      \draw (0,0) node[below]{$z=0$}--(0,4.5);
      \draw (5,0) node[below]{$z=1$}--(5,4.5);
      \draw (0,1.2) node {$\bullet$} node[left]
            {\vbox{\llap{some other root}
                \llap{$\eta_j$ of $\tilde P_\alpha(0,y)$}}};
      \draw (5,1) node {$\bullet$} node[right]
            {\vbox{\rlap{some other root}
                \rlap{$\zeta_j$ of $\tilde P_\alpha(1,y)=p(y)$}}};
      \draw (-.1,.5)--(5.1,.5);
      \draw (-.1,1.5)--(5.1,1.5);
      \draw (5,2.25) node[right]
            {$\kern-6pt\left.\rule{0pt}{2.15em}\right\}>\varepsilon$};
      \begin{scope}[yshift=1cm]
	\draw (0,2.7) node {$\bullet$} node[left] {$\eta_1$};
	\draw (5,2.5) node {$\bullet$} node[right] {$\zeta_1$};
	\draw (-.1,2)--(5.1,2) node[right] {$\zeta_1-\frac\varepsilon2$};
	\draw (-.1,3)--(5.1,3) node[right] {$\zeta_1+\frac\varepsilon2$};
	\draw[smooth] plot coordinates {(0,2.7) (0.5, 2.47695) (1., 2.78779)
          (1.5, 2.33775) (2., 2.4224) (2.5, 2.62738) (3., 2.68165)
          (3.5, 2.15903) (4., 2.48256) (4.5, 2.58423) (5,2.5)};
      \end{scope}
      \draw[smooth] plot coordinates {(0,1.2) (0.5, 0.705543) (1., 0.84889)
        (1.5, 0.927016) (2., 1.35318) (2.5, 1.34782) (3., 0.726507)
        (3.5, 1.25) (4., 0.65396) (4.5, 1.01007) (5,1)};
    \end{tikzpicture}
    \caption{Separation of roots as used in the proof of
      Lemma~\ref{LEM:minpoly}}\label{FIG:balls}
  \end{figure}
	
  Since any two roots of $p$ are separated by more than
  $2\varepsilon$, the distance between any two roots of $\tilde
  P_\alpha(z,y)$ for fixed $z$ with $|z|\leq 1$ is more than
  $\varepsilon$.  We have thus shown $(\ast)$.
	
  Now, let $\eta_1, \dots, \eta_d\in \bar{\set F}$ be the $d$ distinct
  roots of $\tilde P_\alpha(0,y)$, and let their indexing be such that
  $|\eta_j-\zeta_j|<\varepsilon/2$ for each~$j$.  Note that $\tilde
  P_\alpha(0,y)$ is square-free because $|\eta_i-\eta_j|>\varepsilon$
  for $i\neq j$.  This means that $\gcd(\tilde P_\alpha(0,y), (D_y
  \tilde P_\alpha) (0,y)) = 1$.  By $(D_y\tilde P_\alpha)(0,y) =
  p^\prime(y)$, we have $p'(\eta_j) \neq 0$ for all $j$.  It follows
  that
  \[
  \tilde P_\alpha(0,\eta_j) =0\quad \text{and}\quad (D_y \tilde
  P_\alpha)(0,\eta_j) = p^\prime(\eta_j) \neq 0.
  \]
  Applying the implicit function theorem to each $\tilde
  P_\alpha(z,y+\eta_j)\in \set F(\alpha,\eta_j)[z,y]$ (with $\set
  F(\alpha,\eta_j)$ in place of $\set F$) yields that there exist $d$
  distinct formal power series $g_1(z),\dots, g_d(z)$ with each
  $g_j(z)\in \set F(\alpha,\eta_j)[[z]]$ and $g_j(0) = \eta_j$ such
  that $\tilde P_\alpha(z,g_j(z)) = 0$. By $(\ast)$, for each $j$
  there exists a unique integer $k$ with $1\leq k\leq d$ so that
  $g_j(1) = \zeta_k$ and $|g_j(z)-\zeta_k|<\varepsilon/2$ for any $z$
  with $|z|\leq 1$.  Hence $|\eta_j - \zeta_k| < \varepsilon/2$ since
  $\eta_j=g_j(0)$.  By $|\eta_j-\zeta_j|<\varepsilon/2$, we get
  $|\zeta_j-\zeta_k|<\varepsilon$. Thus $j=k$ because any two roots of
  $p$ are separated by more than $2\varepsilon$.
	
  Moreover, all $g_j(z)\in \set F(\alpha,\eta_j)[[z]]$ annihilated by
  $\tilde P_\alpha(z,y)$ are analytic in the disk $|z|\leq 1$.
  Indeed, since the leading coefficient of~$\tilde P_\alpha(z,y)$
  w.r.t.\ $y$ is a nonzero constant, the singularities of the $g_j(z)$
  could only be branch points.  However, the choices of~$\varepsilon$
  and $\delta$ make it impossible for the $g_j(z)$ to have branch
  points in the disk~$|z|\leq 1$, because in order to have a branch
  point, two roots of the polynomial~$\tilde P_\alpha(z,y)$ w.r.t.~$y$
  would need to touch each other as $z$ varies, and we have ensured
  that they are always separated by more than $\varepsilon$ as $z$
  ranges over the unit disk (see $(\ast)$ and Fig.~\ref{FIG:balls}).
	
  Now define the polynomial
  \[
  P_{\alpha}(z,y)= \tilde P_\alpha(z,(1-z)y) = p((1-z) y) -\alpha
  (1-z)\in \set F(\alpha)[z,y].
  \] 
  Observe that for any $g(z)\in \bar{\set F}[[z]]$, we have
  $g(z)/(1-z)$ is a root of~$P_\alpha(z,y)$ if and only if $g(z)$ is a
  root of $\tilde P_\alpha(z,y)$. Thus there exist exactly $d$
  distinct formal power series
  \[
  f_{j}(z)=\frac{g_j(z)}{(1-z)}\in\set F(\alpha,\eta_j)[[z]]\subseteq
  \bar{\set F}[[z]]
  \]
  with $f_{j}(0) =g_j(0) = \eta_j$ and $g_j(1) = \zeta_j$ such that
  $P_\alpha(z,f_j(z))=0$.
	
  Since each $g_j(z)$ is analytic in the disk $|z|\leq 1$ and $g_j(1)
  = \zeta_j\neq 0$, the point $z=1$ is evidently the only singularity
  of $f_{j}(z)$ in the disk $|z|\leq 1$, and thus it is the only
  dominant singularity. In addition, the point $z=1$ is further a
  simple pole of~$f_{j}(z)$ and then
  \[f_{j}(z) \sim \frac{\zeta_j}{1-z} \quad (z\to 1),\] 
  which gives $[z^n] f_{j}(z) \sim \zeta_j~(n\rightarrow \infty)$ by
  part~\ref{THM:transfer} of Theorem~\ref{THM:asympt} (with $\bar{\set F}$
  in place of~$\set F$). Since $\zeta_j\neq 0$, it follows that
  $\lim_{n\rightarrow\infty}[z^n]f_{j}(z) = \zeta_j.$
	
  Further assume that either $\set F\subseteq \set R$ and $\zeta_1\in
  \bar{\set F}\cap\set R$, or $\set F\setminus\set R \neq
  \emptyset$. In either case, $\set F$ is dense in the field $\set
  F(\zeta_1)$ since $\set F\supseteq \set Q$.  Then by the continuity
  of $p$ at $\zeta_1$, with the above $\delta$ and $\varepsilon$, we
  always can find a number $\eta\in \set F$ with
  $|\eta-\zeta_1|<\varepsilon/2$ so that
  $|p(\eta)|=|p(\eta)-p(\zeta_1)|<\delta/2$.  Fix such $\eta\in\set F$
  and let $\alpha = p(\eta) \in \set F$.  Then $\eta$ is a root of
  $\tilde P_\alpha(0,y)$. Since $|\eta_1-\zeta_1|<\varepsilon/2$, we
  have $|\eta_1-\eta|<\varepsilon$. By $(\ast)$ we know $\eta_1 =
  \eta\in \set F$.  The lemma follows by setting $P(z,y)$ to be
  $P_{\alpha}(z,y)$.
\end{proof}

\begin{example}
  The irreducible polynomial $p(y)=y^3-5y^2+3y+2\in\set Q[y]$ has
  three real roots with approximate values $-.39138238063090084510$,
  $1.2271344421706896320$, $4.1642479384602112131$,
  respectively. Consider the polynomial
  \[P(z,y)= p((1-z)y) - p(4)(1-z) \in\set Q[z,y].\]
  This polynomial was found by the construction described in the
  proof, using the initial approximation~$4$.  The equation $P(z,y)=0$
  has a solution
  \[
  f(z)=4+\tfrac{46}{11}z + \tfrac{5538}{1331}z^2 +
  \tfrac{670794}{161051}z^3 + \tfrac{81144794}{19487171}z^4 +
  \tfrac{9819245130}{2357947691}z^5 + \cdots\in\set Q[[z]],
  \]
  the coefficients of which converge to the third root of
  $p(y)$. Note, for example, that the distance of the coefficient of
  $z^4$ to the root is already less than $10^{-4}$.  The other two
  roots of $P(z,y)$ are
  \begin{alignat*}1
    &\tfrac{1}{2}(1-\sqrt{5}) +\tfrac{1}{110}(45-41 \sqrt{5}) z
    +\tfrac{1}{66550}(27925-24377\sqrt{5}) z^2
    +\cdots\in\bar{\set Q}[[z]],\\
    &\tfrac{1}{2}(1+\sqrt{5})
    +\tfrac{1}{110}(45+41 \sqrt{5}) z
    +\tfrac{1}{66550}(27925+24377\sqrt{5}) z^2
    +\cdots\in\bar{\set Q}[[z]].
  \end{alignat*}
  Their coefficient sequences converge to the two other roots
  of~$p(y)$, but do not belong to~$\set Q$.
\end{example}

The following theorem clarifies the converse direction for algebraic
sequences.  It turns out that every element in $\cal A_{\set F}$ is
algebraic over $\set F$. Consequently, $\cal A_{\set F}$ is a field.

\begin{theorem}\label{THM:algnos}
  Let $\set F$ be a subfield of $\set C$.
  \begin{enumerate}
  \item\label{THM:algnos1} If $\set F\subseteq \set R$, then
    $\cal A_{\set F}=\bar{\set F}\cap\set R$.

    \medskip
  \item\label{THM:algnos2} If $\set F\setminus \set R \neq \emptyset$,
    then $\cal A_{\set F} = \bar{\set F}$.
  \end{enumerate}
\end{theorem}
\begin{proof}
  \begin{enumerate}
  \item Let $\xi\in \bar{\set F}\cap \set R$.  Then there exists an
    irreducible polynomial $p(y) \in\set F[y]$ such that $p(\xi)=0$.
    By Lemma~\ref{LEM:minpoly}, $\xi$ is equal to a limit of an
    algebraic sequence over $\set F$, which implies that $\xi \in \cal
    A_{\set F}$.
    
    To show the converse inclusion, we let $\xi \in \cal A_{\set F}$.
    When $\xi = 0$, there is nothing to show. Assume that $\xi \neq
    0$.  Then there is an algebraic sequence $(a_n)_{n=0}^\infty \in
    \set F^{\set N}$ such that $\lim_{n\rightarrow \infty}a_n = \xi$.
    Since $\xi \neq 0$, we have $a_n\sim \xi$ $(n\rightarrow \infty)$.
		
    Let $f(z) = \sum_{n=0}^\infty a_n z^n$. Clearly $f(z)$ is an
    algebraic function over~$\set F$.  By part~\ref{THM:abelian} of
    Theorem~\ref{THM:asympt}, $f(z) \sim \xi/(1-z)$ $(z\rightarrow
    1^-)$.  Since $f(z)$ is algebraic, there exists a positive integer
    $s$ such that $f(z)$ admits a Puiseux expansion
    \[
    f(z) = \frac{\xi}{1-z} + \sum_{n=-s+1}^\infty b_n(1-z)^{n/s} \quad
    \text{with} \ b_n \in \set C\ \text{for all}\ n.
    \]
    Setting $g(z) = f(z)(1-z)$ establishes that
    \[g(z) = \xi + \sum_{n=-s+1}^\infty b_n (1-z)^{n/s+1}.\]
    Note that $n/s + 1 > 0$, so $g(z)$ is finite at $z =1$. Sending
    $z$ to one gives $g(1)=\xi$.  By closure properties, $g(z)$ is
    again an algebraic function over $\set F\subseteq \set R$.  Thus
    $\xi = g(1)\in \bar{\set F}\cap \set R$.
    
  \item 
    By Lemma~\ref{LEM:minpoly} and a similar argument as above, we
    have $\cal A_{\set F} = \bar{\set F}$.
  \end{enumerate}
\end{proof}

If we were to consider the class $\cal C_{\set F}$ of limits of
convergent sequences in $\set F$ satisfying linear recurrence
equations with constant coefficients over~$\set F$, sometimes called
C-finite sequences, then an argument analogous to the above proof
would imply that $\cal C_{\set F}\subseteq\set F$, because the power
series corresponding to such sequences are rational functions, and the
values of rational functions over~$\set F$ at points in~$\set F$
evidently gives values in~$\set F$.  The converse direction $\set F
\subseteq\cal C_{\set F}$ is trivial, so we have $\cal C_{\set F}=\set F$.

\begin{corollary}\label{COR:qcase}
  If $\set F\subseteq \set R$, then $\bar{\set F} = \cal A_{\set F(i)}
  = \cal A_{\set F}[i] =\cal A_{\set F} + i \cal A_{\set F}$, where
  $i$ is the imaginary unit.
\end{corollary}
\begin{proof}
  Since $\cal A_{\set F}$ is a ring and $i^2 = -1\in\set F\subseteq
  \cal A_{\set F}$, we have $\cal A_{\set F}[i] = \cal A_{\set F} + i
  \cal A_{\set F}$.  Since $i\in \bar{\set F}$ and~$\set F\subseteq
  \set R$, the field $\bar{\set F}$ is closed under complex
  conjugation and then
  \[
  \bar{\set F} = (\bar{\set F}\cap\set R) + i(\bar{\set F}\cap \set R)
  = \cal A_{\set F} + i\cal A_{\set F},
  \]
  by part~\ref{THM:algnos1} of Theorem~\ref{THM:algnos}.  It follows
  from part~\ref{THM:algnos2} of Theorem~\ref{THM:algnos} that
  $\cal A_{\set F(i)}=\overbar{\set F(i)}$.  Since $\cal A_{\set F}
  \subseteq \cal A_{\set F(i)}$ and $i\in \cal A_{\set F(i)}$,
  \[
  \bar{\set F}= \cal A_{\set F} + i \cal A_{\set F}\subseteq
  \cal A_{\set F(i)}=\overbar{\set F(i)}=\bar{\set F},
  \]
  The assertion holds.
\end{proof}

The following lemma says that every element in $\bar{\set F}$ can be
written as the value at one of an analytic algebraic function
vanishing at zero, provided that $\set F$ is dense in~$\set C$.  This
will be used in the next section to restrict the evaluation domain.

\begin{lemma}\label{LEM:algseq}
  Let $\set F$ be a subfield of $\set C$ with $\set F\setminus \set R
  \neq \emptyset$.  Let $p(y)\in \set F[y]$ be an irreducible
  polynomial of degree $d$.  Assume that $\zeta_1, \dots, \zeta_d$ are
  all the (distinct) roots of $p$ in $\bar{\set F}$.  Then there is a
  polynomial $P(z,y)\in \set F[z,y]$ of degree~$d$ in~$y$ admitting
  $d$ distinct roots $g_1(z)\in \set F[[z]]$ and $g_2(z), \dots,
  g_d(z)\in \bar{\set F}[[z]]$ such that all $g_j(z)$ are analytic in
  the disk $|z|\leq 1$ with $g_j(0)=0$ and, after reordering (if
  necessary), $g_j(1)=\zeta_j$.
\end{lemma}
\begin{proof}
  If $d=1$ then $p(y) = y-\zeta_1$ with $\zeta_1 \in \set F$.  Letting
  $P(z,y) =y-\zeta_1 z$ yields the assertion.  Otherwise $d>1$ and all
  roots $\zeta_1, \dots, \zeta_d$ are nonzero.
	
  By Lemma~\ref{LEM:minpoly}, there exists a polynomial $\tilde P
  (z,y)$ in~$\set F[z,y]$ of degree~$d$ in~$y$ admitting $d$ distinct
  roots $f_1(z)\in \set F[[z]]$ and $f_2(z),\dots, f_d(z) \in
  \bar{\set F}[[z]]$ such that each $f_j(z)$ is analytic in the disk
  $|z|\leq 1$ except for a simple pole at $z=1$ and, after reordering
  (if necessary),
  \[\lim_{n\rightarrow \infty}[z^n]f_j(z) = \zeta_j, \quad j = 1, \dots d.\]
  Hence, together with part~\ref{THM:abelian} of Theorem~\ref{THM:asympt},
  each $f_j(z)$ admits an expansion at $z=1$ of the form
  \[ f_j(z)  \sim  \frac{\zeta_j}{1-z} \quad(z\to 1^-).\]
  For each~$j$ set $g_j(z) = f_j(z)z(1-z)$.  Then $g_1(z)\in \set
  F[[z]]$, $g_2(z), \dots, g_d(z)\in \bar{\set F}[[z]]$ and they are
  distinct from each other.  Moreover, each $g_j(z)$ is analytic in
  the disk $|z|\leq 1$ with $g_j(0)=0$ and $g_j(1)=\zeta_j$.  By
  closure properties, all $g_j(z)$ are again algebraic functions.
  Define
  \[P(z,y) = \prod_{j=1}^d(y-g_j(z)) 
  = \prod_{j=1}^d \left(y-f_j(z) z(1-z)\right)\in \overbar{\set F(z)}[y].\] 
  Since the coefficients of $P(z,y)$ w.r.t.\ $y$ are symmetric in the
  conjugates $f_1(z),\dots,f_d(z)$, they all belong to the field~$\set
  F(z)$. Multiplying $P$ by a suitable polynomial in $\set F[z]$ gives
  a desired polynomial in $\set F[z,y]$.  The lemma follows.
\end{proof}

\section{D-finite Numbers}

Let us now return to the study of D-finite numbers.  Let $R$ be a
subring of~$\set C$ and $\set F$ be a subfield of $\set C$.  Recall
that by Definition~\ref{DEF:dfinitenos}, the elements of $\cal D_{R,\set F}$
are exactly limits of convergent sequences in $R^{\set N}$ which are
P-recursive over~$\set F$.  Some facts about P-recursive sequences
translate directly into facts about $\cal D_{R,\set F}$.

\begin{proposition}\label{PROP:property}\leavevmode\null
  \begin{enumerate}
  \item\label{PROP:0}
    $R\subseteq\cal D_{R,\set F}$ and $\cal A_{\set F}\subseteq\cal D_{\set F}$.

    \medskip
  \item\label{PROP:1}
    If $R_1\subseteq R_2$ then $\cal D_{R_1,\set F}\subseteq\cal D_{R_2,\set F}$,
    and if $\set F\subseteq\set E$ then $\cal D_{R,\set F}\subseteq
    \cal D_{R,\set E}$.

    \medskip
  \item\label{PROP:2}
    $\cal D_{R,\set F}$ is a subring of~$\set C$.  Moreover, if $R$ is
    an $\set F$-algebra then so is $\cal D_{R,\set F}$.

    \medskip
  \item\label{PROP:3}
    If $\set E$ is an algebraic extension field of~$\set F$, then
    $\cal D_{R,\set F}=\cal D_{R,\set E}$.

    \medskip
  \item\label{PROP:4}
    If $R\subseteq\set F$, then $\cal D_{R,\set F}=\cal D_{R,\quot(R)}$.

    \medskip
  \item\label{PROP:5}
    If $R$ and $\set F$ are closed under complex conjugation, then so
    is $\cal D_{R,\set F}$.
    In this case, we have $ \cal D_{R,\set F}\cap \set R
    =\cal D_{R\cap \set R, \set F}$.
    Moreover, if the imaginary unit $i\in \cal D_{R,\set F}$ then
    $\cal D_{R,\set F} =\cal D_{R\cap \set R, \set F} + i \cal D_{R\cap \set R, \set F}$.
  \end{enumerate}
\end{proposition}
\begin{proof}
  \begin{enumerate}
  \item The first inclusion is clear because every element of $R$ is
    the limit of a constant sequence, and every constant sequence is
    P-recursive. The second inclusion follows from the fact that
    algebraic functions are D-finite, and the coefficient sequences of
    D-finite functions are P-recursive.

    \medskip
  \item Clear.

    \medskip
  \item Follows directly from the corresponding closure properties for
    P-recursive sequences.

    \medskip
  \item Follows directly from part~\ref{lemma:changeopseq1} of
    Lemma~\ref{lemma:changeopseq}.

    \medskip
  \item Follows directly from part~\ref{lemma:changeopseq2} of
    Lemma~\ref{lemma:changeopseq}.

    \medskip
  \item For any convergent sequence $(a_n)_{n=0}^\infty\in R^{\set N}$, we have
    \[
    \re\left(\lim_{n\rightarrow \infty}a_n\right)=\lim_{n\rightarrow \infty}\re(a_n),\
    \im\left(\lim_{n\rightarrow \infty}a_n\right)=\lim_{n\rightarrow \infty}\im(a_n),
    \]
    and thus $\overbar{\lim_{n\rightarrow \infty}a_n}
    = \lim_{n\rightarrow \infty}\bar{a}_n$.
    Hence the first assertion follows by $(\bar a_n)_{n=0}^\infty\in
    R^{\set N}$ and part~\ref{lemma:changeopseq3} of
    Lemma~\ref{lemma:changeopseq}.    
    Since $R$ is closed under complex conjugation,
    $(\re(a_n) )_{n=0}^\infty\in (R\cap \set R)^{\set N}$.  Then the inclusion
    $\cal D_{R,\set F} \cap \set R\subseteq \cal D_{R \cap \set R, \set F}$
    can be shown similarly as the first assertion. The
    converse direction holds by part~\ref{PROP:1}.
    Therefore $ \cal D_{R,\set F}\cap \set R=\cal D_{R\cap \set R, \set F}$.
    Moreover, if $i$ belongs to $\cal D_{R,\set F}$,
    then $\cal D_{R\cap \set R,\set F} +
    i \cal D_{R\cap \set R,\set F}\subseteq \cal D_{R,\set F}$ since
    $\cal D_{R\cap \set R,\set F} \subseteq \cal D_{R,\set F}$.  To
    show the converse inclusion, let $\xi \in \cal D_{R,\set F}$.
    Then $\overbar{\xi}\in \cal D_{R,\set F}$ by the first
    assertion. Since $i \in \cal D_{R,\set F}$ and $R$ is closed under
    complex conjugation, $\re(\xi),\im(\xi)$ both belong to $\cal
    D_{R,\set F}\cap\set R = \cal D_{R\cap \set R,\set F}$ by the
    second assertion. Therefore $\xi=\re(\xi)+i\im(\xi)\in
    \cal D_{R\cap \set R,\set F} + i \cal D_{R\cap \set R,\set F}$.
  \end{enumerate}
\end{proof}
\goodbreak
\begin{example}\leavevmode\null
  \begin{enumerate}
  \item We have $\cal D_{\set Q(\sqrt2),\set Q(\pi,\sqrt2)}= \cal D_{\set
    Q(\sqrt2),\set Q(\sqrt2)}=\cal D_{\set Q(\sqrt2),\set Q}$.  The
    first identity holds by part~\ref{PROP:4}, the second by
    part~\ref{PROP:3} of the proposition.

    \medskip
  \item We have $\cal D_{\bar{\set Q},\set Q}=\cal D_{\bar{\set Q},\set R}$.
    The inclusion ``$\subseteq$'' is clear by part~\ref{PROP:1}.
    For the inclusion ``$\supseteq$'', let~$\xi\in\cal D_{\bar{\set Q},\set R}$.
    Then $\xi=a+ib$ for some $a,b\in\set R$,
    and there exists a sequence $(a_n+ib_n)_{n=0}^\infty$ in~$\bar{\set Q}^{\set N}$
    and an operator $L\in\set R[n]\<S_n>$ such that $L\cdot (a_n+ib_n)=0$
    and $\lim_{n\to\infty}(a_n+ib_n)=a+ib$.
    Since the coefficients of $L$ are real, we then have $L\cdot a_n=0$
    and $L\cdot b_n=0$.
    Furthermore, $\lim_{n\to\infty}a_n=a$ and $\lim_{n\to\infty}b_n=b$.
    Therefore,
    \[
    a,b\in\cal D_{\bar{\set Q}\cap\set R,\set R}
    \stackrel{\text{part~\ref{PROP:4}}}=
    \cal D_{\bar{\set Q}\cap\set R,\bar{\set Q}\cap\set R}
    \stackrel{\text{part~\ref{PROP:3}}}=
    \cal D_{\bar{\set Q}\cap\set R,\set Q}.
    \]
    Hence $a+ib\in\cal D_{\bar{\set Q}\cap\set R,\set Q}+
    i\cal D_{\bar{\set Q}\cap\set R,\set Q}
    \stackrel{\text{part~\ref{PROP:5}}}=\cal D_{\bar{\set Q},\set Q}$,
    as claimed.
  \end{enumerate}
\end{example}

The results for C-finite and algebraic cases motivate the following
theorem, which says that every D-finite number is essentially the
(left) limiting value at one of a D-finite function.

\begin{theorem}\label{THM:eval}
  Let $R$ be a subring of $\set C$ and let $\set F$ be a subfield of
  $\set C$.  Then for every~$\xi \in \cal D_{R,\set F}$, there exists
  $g(z) \in R[[z]]$ D-finite over $\set F$ such that $\xi =\lim_{z\to 1^{-}} g(z)$.
\end{theorem}
\begin{proof}
  The statement is clear when $\xi = 0$. Assume that $\xi$ is nonzero.
  Then there exists a sequence $(a_n)_{n=0}^\infty\in R^{\set N}$
  P-recursive over $\set F$ such that $\lim_{n\rightarrow \infty} a_n
  = \xi$.  Since $\xi \neq 0$, we have $a_n \sim \xi $ $(n \rightarrow
  \infty)$.  Let $f(z)=\sum_{n=0}^\infty a_nz^n$.  By
  part~\ref{THM:abelian} of Theorem~\ref{THM:asympt},
  $f(z)\sim \xi/(1-z)$ as $z\rightarrow 1^{-}$, which, by definition,
  implies that
  \[
  \lim_{z\to 1^{-}} \frac{f(z)}{\xi / (1-z)} = \lim_{z\to 1^{-}}
  \frac{(1-z) f(z)}{\xi} = 1.
  \]
  Letting $g(z)=f(z)(1-z)$ gives $\lim_{z\to 1^{-}} g(z)/\xi = 1$, and
  then $\lim_{z\to 1^{-}} g(z) = \xi$ since $\xi \neq 0$.  The
  assertion follows by noticing that $g(z)\in R[[z]]$ is D-finite over
  $\set F$ due to closure properties.
\end{proof}

\begin{example}\label{EX:evaluation}
  We have $\zeta(3)=\sum_{n=1}^\infty\frac1{n^3}=\lim_{z\to 1^{-}}
  \Li_3(z) = \Li_3(1)$, where
  $\Li_3(z)=\sum_{n=1}^\infty\frac{1}{n^3}z^n$ is the polylogarithm
  function. Note that $\Li_3(z)\in\set Q[[z]]$ is D-finite over
  $\set Q$ and finite at $z=1$.
\end{example}

Note that the above theorem implies that D-finite numbers are
computable when the ring~$R$ and the field~$\set F$ consist of
computable numbers.  This allows the construction of (artificial)
numbers that are not D-finite.

We next turn to some sort of converse of Theorem~\ref{THM:eval}. To
this end, we need to develop several lemmas. First note that we may
assume without loss of generality that $\set F\setminus\set R\neq\emptyset$,
because for any $\set F$ we will always have $\set F(i)\setminus
\set R\neq\emptyset$ and, by part~\ref{PROP:3} of Prop.~\ref{PROP:property},
$\cal D_{R,\set F}=\cal D_{R,\set F(i)}$, so we can always replace $\set F$
by $\set F(i)$. Let us thus assume $\set F\setminus\set R\neq\emptyset$
for the remainder of this section.

The first lemma says that the value of a D-finite function at
any non-singular point in~$\bar{\set F}$ can be represented by 
the value of another D-finite function at one.
\begin{lemma}\label{LEM:evalFbar}
  Let $\set F$ be a subfield of $\set C$ with $\set F\setminus \set R
  \neq \emptyset$ and $R$ be a subring of $\set C$ containing~$\set F$.
  Assume that $f(z)\in \cal D_{R,\set F}[[z]]$ is analytic at zero and
  annihilated by a nonzero operator $L \in \set F[z]\langle D_z\rangle$
  with zero being an ordinary point.  Then for any non-singular
  point $\zeta \in {\bar{\set F}}$ of $L$, there exists
  $h(z)\in \cal D_{R,\set F}[[z]]$ and $M\in \set F[z]\langle D_z\rangle$
  nonzero with zero and one being ordinary points such that
  $M\cdot h(z) = 0$ and $f(\zeta)=h(1)$.
\end{lemma}
\begin{proof}
  Let $\zeta_1 \in \bar{\set F}$ be a non-singular point of~$L$.  Then
  there exists an irreducible polynomial $p(z)\in \set F[z]$ of degree~$d$
  such that $p(\zeta_1)=0$.  Let~$\zeta_2, \dots, \zeta_d$ be all
  other roots of $p$ in $\bar{\set F}$.  By Lemma~\ref{LEM:algseq},
  there exists a polynomial $P(z,y)\in \set F[z,y]$ of degree~$d$
  in~$y$ admitting $d$ distinct roots $g_1(z)\in \set F[[z]]$ and
  $g_2(z), \dots, g_d(z)\in \bar{\set F}[[z]]$ such that all $g_j(z)$
  are analytic in the disk $|z|\leq 1$ with $g_j(0)=0$ and
  $g_j(1)=\zeta_j$.  In particular, all $g_j(z)$ are analytic in a
  neighborhood of zero including the point $z=1$.
	
  Since $g_1(1) =\zeta_1$ is not a singularity of $L$ by assumption,
  none of $g_j(1)=\zeta_j$ is a singularity of~$L$.  In fact, suppose
  otherwise that for some $2\leq \ell \leq d$, the point $g_\ell(1) =
  \zeta_\ell$ is a root of~$\lc(L)$.  Since $\lc(L)\in \set F[z]$ and
  $p$ is the minimal polynomial of~$\zeta_\ell$ over~$\set F$, we know
  that $p(z)$ divides $\lc(L)$ over~$\set F$.  Thus $\zeta_1$ is also
  a root of $\lc(L)$, a contradiction.
	
  By assumption, zero is an ordinary point of $L$. Note that $g_j(0) =
  0$ for all $j$.  It follows from Theorem~\ref{THM:algsubs} that
  there exists a nonzero operator $M\in \set F[z]\langle D_z\rangle$
  with $M\cdot(f\circ g_1) =0$ which does not have zero or one among
  its singularities.  By part~\ref{PROP:0} of
  Proposition~\ref{PROP:property}, $\set F\subseteq R\subseteq
  \cal D_{R,\set F}$.  Since $f(z) \in \cal D_{R,\set F}[[z]]$
  and~$g_1(z)\in \set F[[z]]$ with $g_1(0)=0$, we have $f(g_1(z))\in
  \cal D_{R,\set F}[[z]]$.  Setting $h(z) = f(g_1(z))$ completes the
  proof.
\end{proof}

With the above lemma, it suffices to consider the case when the
evaluation point is in $R\cap \set F$.  This is exactly what the next
two lemmas are concerned about.

\begin{lemma}\label{LEM:coeffRinsideconverg}
  Assume that $f(z) = \sum_{n=0}^\infty a_n z^n \in R[[z]]$ is
  D-finite over $\set F$ and convergent in some neighborhood of zero.
  Let $\zeta \in R\cap\set F$ be in the disk of convergence.  Then
  $f^{(k)}(\zeta) \in \cal D_{R,\set F}$ for all~$k \in \set N$.
\end{lemma}
\begin{proof}
  For any $k\in \set N$, it is well-known that $f^{(k)}(z)\in R[[z]]$
  is also D-finite and has the same radius of convergence at zero as
  $f(z)$.  Note that since $f(z)$ is D-finite over~$\set F$, so is
  $f^{(k)}(z)$.  Thus to prove the lemma, it suffices to show the case
  when $k=0$, i.e., $f(\zeta)\in \cal D_{R,\set F}$.
	
  Since $f(z)$ is D-finite over $\set F$, the coefficient sequence
  $(a_n)_{n=0}^\infty$ is P-recursive over~$\set F$.  Note that $\zeta
  \in R\cap\set F$ is in the disk of convergence of $f(z)$ at zero, so
  \[
  f(\zeta) = \sum_{n=0}^\infty a_n \zeta^n = \lim_{n\rightarrow \infty}
  \sum_{\ell = 0}^n a_\ell \zeta^\ell.
  \]
  Since $(\zeta^n)_{n=0}^\infty$ is P-recursive over~$\set F$, the
  assertion follows by noticing that $(\sum_{\ell = 0}^n a_\ell
  \zeta^\ell)_{n=0}^\infty\in R^{\set N}$ is P-recursive over $\set F$
  due to closure properties.
\end{proof}

\begin{example}\leavevmode\null
  \begin{enumerate}
  \item Since $\exp(z)=\sum_{n=0}^\infty\frac1{n!}z^n\in\set Q[[z]]$
    is D-finite over~$\set Q$, and converges everywhere, we get from
    the lemma that the numbers $\e, 1/\e, \sqrt{\e}$ belong to~$\cal
    D_{\set Q,\set Q}$.  More precisely, since we are currently only
    considering non-real fields~$\set F$, we could say that $\exp(z)$
    is D-finite over $\bar{\set Q}$, therefore
    $\e,1/\e,\sqrt{\e}\in\cal D_{\set Q,\bar{\set Q}}$, but by
    Proposition~\ref{PROP:property}, $\cal D_{\set Q,\bar{\set Q}}
    =\cal D_{\set Q,\set Q}$.

    \medskip
  \item All (finite) values of G-function at algebraic numbers belong
    to $\cal D_{\set Q(i)}$, as remarked in the introduction. Indeed,
    \cite[Theorem~1]{FiRi2014} tells us that any complex number $\xi$
    that appears as the value of a G-function at some algebraic number
    has real and imaginary parts both of the form $f(1)$ for some
    G-function $f(z)$ with rational coefficients and whose radius of
    convergence is greater than one. Together with the above lemma, we
    readily see that such $\xi$ belong to $\cal D_{\set Q, \set Q(i)}
    + i \cal D_{\set Q, \set Q(i)}$, which is actually equal to
    $\cal D_{\set Q(i)}$ by part~\ref{PROP:5} of
    Proposition~\ref{PROP:property}.
  \end{enumerate}
\end{example}
\begin{lemma}\label{LEM:coeffDinsideconverg}
  Let $R$ be a subring of $\set C$ containing $\set F$.  Let $f(z) =
  \sum_{n=0}^\infty a_nz^n\in \cal D_{R,\set F}[[z]]$ be an analytic
  function at zero.  Assume that there exists a nonzero operator $L
  \in \set F[z]\langle D_z\rangle$ with zero being an ordinary point such
  that $L\cdot f(z) =0$.  Let~$r>0$ be the smallest modulus of roots
  of $\lc(L)$ and let $\zeta \in\set F$ with~$|\zeta|<r$.  Then
  $f^{(k)}(\zeta) \in \cal D_{R,\set F}$ for all $k \in \set N$.
\end{lemma}
\begin{proof}
  Let $\rho$ be the order of $L$. Since zero is an ordinary point of
  $L$, there exist P-recursive sequences $(c_n^{(0)})_{n=0}^\infty,
  \dots, (c_n^{(\rho-1)})_{n=0}^\infty$ in $\set F^{\set N}\subseteq
  R^{\set N}$ with $c_j^{(m)}$ equal to the Kronecker delta
  $\delta_{mj}$ for $m, j= 0,\dots, \rho-1$, so that the set
  $\{\sum_{n=0}^\infty c_n^{(m)} z^n \}_{m=0}^{\rho -1}$ forms a basis
  of the solution space of~$L$ near zero. Since $L\cdot f(z) =0$ and
  $f(z)$ is analytic at zero,
  \begin{equation}\label{EQ:linearcomb}
    f(z)=\sum_{n=0}^\infty a_n z^n = a_0 \sum_{n=0}^\infty c_n^{(0)}
    z^n + \dots + a_{\rho-1}\sum_{n=0}^\infty c_n^{(\rho-1)}z^n
  \end{equation}
  near zero.  Note that the singularities of solutions of~$L$ can only
  be roots of $\lc(L)$.  Hence $f(z)$ as well as $\sum_{n=0}^\infty
  c_n^{(m)} z^n$ for $m = 0, \dots, \rho-1$ are convergent in the disk
  $|z|<r$.  It follows from $|\zeta|<r$ and
  Lemma~\ref{LEM:coeffRinsideconverg} that the set
  $\{\sum_{n=0}^\infty c_n^{(m)} \zeta^n\}_{m=0}^{\rho -1}$ belongs to
  $\cal D_{R,\set F}$. Since $a_0,\dots,a_{\rho-1}\in \cal D_{R,\set F}$,
  letting $z = \zeta$ in Eq.~\eqref{EQ:linearcomb} yields that
  $f(\zeta)$ is D-finite by closure properties.  Differentiating
  Eq.~\eqref{EQ:linearcomb}, we find by
  Lemma~\ref{LEM:coeffRinsideconverg} that for $k>0$, the derivative
  $f^{(k)}(\zeta)$ also belongs to $\cal D_{R,\set F}$.
\end{proof}
\goodbreak
\begin{example}\leavevmode\null\label{EX:inside}
  \begin{enumerate}
  \item
    We know from Proposition~\ref{PROP:property} that $\sqrt{2}\in\cal
    D_{\set Q}$.  The series
    \[
    (z+1)^{\sqrt2}=1 + \sqrt2z+(1-\frac1{\sqrt2})z^2 + \cdots\in
    \set Q(\sqrt2)[[z]]\subseteq\cal D_{\set Q}[[z]]
    \]
    is D-finite over~$\set Q$, an annihilating operator is
    $(z+1)^2D_z^2+(z+1)D_z-2$.  Here we have the radius $r=1$. Taking
    $\zeta=\sqrt2-1$, the lemma implies $\sqrt2^{\sqrt2}\in\cal D_{\set Q}$.

    \medskip
  \item\label{EX:bessel}
    Observe that the lemma refers to the singularities of the operator
    rather than to the singularities of the particular solution at
    hand. For example, it does not imply that $J_1(1)\in\cal D_{\set Q,\set Q}$,
    where $J_1(z)$ is the first Bessel function, because
    its annihilating operator is $z^2D_z^2+zD_z+(z^2-1)$, which has a
    singularity at zero.  It is not sufficient that the particular
    solution $J_1(z)\in\set Q[[z]]$ is analytic at zero.  Of course, in
    this particular example we see from the series representation
    $J_1(1)=\frac12\sum_{n=0}^\infty \frac{(-1/4)^n}{(n+1)n!^2}$ that
    the value belongs to~$\cal D_{\set Q,\set Q}$.

    \medskip
  \item
    The hypergeometric function
    $f(z):={_2F_1}(\frac13,\frac12,1,z+\frac12)$ can be viewed as an
    element of~$\cal D_{\set Q,\set Q}[[z]]$:
    \begin{alignat*}1
      f(z) &= \sum_{n=0}^\infty
      \underbrace{\left[\frac1{2^nn!}\sum_{k=n}^\infty
          \frac{(1/2)_k(1/3)_k}{k!(k-n)!}
          \Bigl(\frac12\Bigr)^k\right]}_{\in\cal D_{\set Q,\set Q}}z^n.
    \end{alignat*}
    The function $f$ is annihilated by the operator
    \[
    L = 3(2z-1)(2z+1)D_z^2 +(22z-1)D_z + 2.
    \]
    This operator has a singularity at $z=1/2$, and there is no
    annihilating operator of $f$ which does not have a singularity
    there. Although
    $f(1/2)=\frac{\Gamma(1/6)}{\Gamma(1/2)\Gamma(2/3)}$ is a finite
    and specific value, the lemma does not imply that this value is a
    D-finite number.
  \end{enumerate}
\end{example}

\begin{theorem}\label{THM:coeffDoutsideconverg}
  Let $\set F$ be a subfield of $\set C$ with $\set F\setminus
  \set R\neq \emptyset$ and let $R$ be a subring of~$\set C$
  containing~$\set F$.  Assume that $f(z) \in \cal D_{R,\set F}[[z]]$
  is analytic at zero and there exists a nonzero operator $L \in
  \set F[z]\langle D_z\rangle$ with zero being an ordinary point so that
  $L\cdot f(z) =0$.  Further assume that~$\zeta \in\bar{\set F}$ is
  not a singularity of~$L$.  Then $f^{(k)}(\zeta) \in \cal D_{R,\set F}$
  for all $k \in \set N$.
\end{theorem}
\begin{proof}
  By Lemma~\ref{LEM:evalFbar}, it suffices to show that the assertion
  holds for $\zeta = 1$ (or more generally $\zeta \in \set F$).  Now
  assume that $\zeta \in \set F$.  We apply the method of analytic
  continuation.
	
  Let $\mathcal{P}$ be a simple path with a finite cover
  $\bigcup_{j=0}^s \cal B_{r_j}(\beta_j)$, where $s\in \set N$,
  $\beta_0=0$, $\beta_s = \zeta$, $\beta_j \in \set F$, $r_j>0$ is the
  distance between~$\beta_j$ and the zero set of~$\lc(L)$, and $\cal
  B_{r_j}(\beta_j)$ is the open ball centered at~$\beta_j$ and with
  radius $r_j$.  Moreover, $\beta_{j+1}\in \cal B_{r_j}(\beta_j)$ for
  each~$j$ (as illustrated by Fig.~\ref{FIG:ac}).  Such a path exists
  because $\set F$ is dense in $\set C$ and the zero set of~$\lc(L)$
  is finite.  Since the path $\mathcal{P}$ avoids all roots of
  $\lc(L)$, the function $f(z)$ is analytic along $\mathcal{P}$.  We
  next use induction on the index $j$ to show that~$f^{(k)}(\beta_j)
  \in \cal D_{R,\set F}$ for all $k \in \set N$.
  \begin{figure}[ht]
    \centering
    \begin{tikzpicture}[scale=1.3]
      \draw[dashed,thick](0,0) circle (1); 
      \fill (0,0) circle (2pt) node[below right] {$\beta_0=0$}; 
      \fill (0.5,.75) circle (2pt) node[below right] {$\beta_1$}; 
      \draw[dashed,](0.5,0.75) circle (.7); 
      \fill (1,1) circle (2pt) node[below right] {$\beta_2$};
      \draw[dashed](1,1) circle (.65); 
      \fill (1.5,.915) circle (2pt) node[below right] {$\beta_3$}; 
      \fill (2.8,1.76) circle (2pt) node[below right] {$\beta_{s-2}$}; 
      \draw[dashed](2.8,1.76) circle (.6); 
      \fill (3,2.04) circle (2pt) node[below right] {$\beta_{s-1}$}; 
      \draw[dashed](3,2.04) circle (.6); 
      \fill (3.2,2.36) circle (2pt) node[below right] {$\beta_s =\zeta$};
      \node at (2.3,.8) {$\mathcal{P}$}; 
      \draw[->] (0,0)--(-.8,-.6);  
      \node at (-.6,-.3) {$r_0$}; 
      \draw[->] (.5,.75)--(-.2,.75); 
      \node at (.2,.55) {$r_1$}; 
      \draw[->] (1,1)--(1.3,1.57663); 
      \node at (1.4,1.3) {$r_2$}; 
      \draw[->] (2.8,1.76)--(2.2,1.76); 
      \node at (1.9,1.8) {$r_{s-1}$}; 
      \draw[->] (3,2.04)--(2.8,2.60569); 
      \node at (2.9,2.8) {$r_s$}; 
      \draw[thick, domain = 0:1.2] plot (\x,{1-(\x-1)^2}); 
      \draw[thick,domain=1.2:1.9] plot (\x,{1/2*(\x-1.5)^2+.915}); 
      \draw[thick,dashed,domain=1.9:2.5] plot (\x, {1/2*(\x-1.5)^2+.915}); 
      \draw[thick,domain=2.5:3.2] plot (\x, {1/2*(\x-1.5)^2+.915}); 
      \draw[thick] (-1.5,0)--(3,0);
      \draw[thick] (0,-1.5)--(0,3); 
      \draw plot[mark = x,mark options = {color = magenta,scale=1.3}] (1,0); 
      \draw plot[mark = x,mark options = {color = magenta,scale=1.3}] (-.8,.6); 
      \draw plot[mark = x,mark options = {color = magenta,scale=1.3}]
      (.530329,1.44934); 
      \draw plot[mark = x,mark options = {color = magenta,scale=1.3}] (-.7,-1);
      \draw plot[mark = x,mark options = {color = magenta,scale=1.3}] (1,-.6); 
      \draw plot[mark = x,mark options = {color = magenta,scale=1.3}] (1.5,.4); 
      \draw plot[mark = x,mark options = {color = magenta,scale=1.3}]
      (2.43226,2.23409); 
      \draw plot[mark = x,mark options = {color = magenta,scale=1.3}] (2.8,.9);
    \end{tikzpicture}
    \caption{a simple path $\mathcal{P}$ with a finite cover
      $\bigcup_{j=0}^s \cal B_{r_j}(\beta_j)$
      ({\protect\tikz\protect\draw plot[mark = x,mark options = {color
            = magenta,scale=1.3}] (0,0);} stands for the roots of
      $\lc(L)$)}\label{FIG:ac}
  \end{figure}
  
  It is trivial when $j=0$ as $f^{(k)} (\beta_0) =f^{(k)} (0)\in
  \cal D_{R,\set F}$ for $k\in \set N$ by assumption.  Assume now that
  $0< j\leq s$ and $f^{(k)}(\beta_{j-1})\in \cal D_{R,\set F}$ for all $k
  \in \set N$.  We consider $f(\beta_{j})$ and its derivatives.
	
  Recall that $r_{j-1}>0$ is the distance between $\beta_{j-1}$ and
  the zero set of $\lc(L)$.  Since~$f(z)$ is analytic
  at~$\beta_{j-1}$, it is representable by a convergent power series
  expansion
  \[f(z) = \sum_{n=0}^\infty \frac{f^{(n)}(\beta_{j-1})}{n!}(z-\beta_{j-1} )^n 
  \quad \text{for all} \ |z-\beta_{j-1}|<r_{j-1}.\] 
  By the induction hypothesis, $f^{(n)}(\beta_{j-1})/n! \in \cal D_{R,\set F}$ 
  for all $n\in \set N$ and thus $f(z) \in \cal D_{R,\set F}[[z-\beta_{j-1}]]$. 
  Let $Z = z-\beta_{j-1}$, i.e., $z = Z+\beta_{j-1}$.  
  Define $g(Z) = f(Z+\beta_{j-1})$ and $\tilde{L}$ to be the operator 
  obtained by replacing $z$ in $L$ by $Z+\beta_j$. 
  Since $\beta_{j-1}\in \set F\subseteq \cal D_{R,\set F}$ and~$D_z = D_Z$,
  we have $g(Z)\in \cal D_{R,\set F}[[Z]]$ and
  $\tilde{L} \in \set F[Z]\langle D_Z\rangle$.  
  Note that~$L\cdot f(z) =0$ and $\beta_{j-1}$ is an ordinary point
  of~$L$ as $r_{j-1}>0$.
  It follows that $\tilde{L} \cdot g(Z) = 0$ and zero is an ordinary
  point of~$\tilde{L}$.  Moreover, we see that $r_{j-1}$ is now the
  smallest modulus of roots of $\lc(\tilde{L})$.  
  Since $|\beta_{j} - \beta_{j-1}| < r_{j-1}$, 
  applying Lemma~\ref{LEM:coeffDinsideconverg} to $g(Z)$ yields
  $f^{(k)}(\beta_{j}) =g^{(k)}(\beta_{j}-\beta_{j-1}) \in \cal D_{R,\set F}$
  for $k\in \set N$.  
  Thus the assertion holds for~$j = s$. The theorem follows.
\end{proof}

\begin{example}
  By the above theorem, $\exp(\sqrt2)$ and $\log(1+\sqrt3)$ both
  belong to~$\cal D_{\set Q}$.  We also have $\e^\pi \in \cal D_{\set Q}$.
  This is because $\e^\pi = (-1)^{-i}$, with $i$ the imaginary
  unit, is equal to the value of the D-finite function $(z+1)^{-i}\in
  \set Q(i)[[z]]$ at $z=-2$ (outside the radius of convergence;
  analytically continued in consistency with the usual branch cut
  conventions) and then $\e^\pi\in \cal D_{\set Q(i)} \cap \set R =
  \cal D_{\set Q}$.
\end{example}

\section{Open Questions}\label{SEC:conclusion}

We have introduced the notion of D-finite numbers and made some first
steps towards understanding their nature. We believe that, similarly
as for D-finite functions, the class is interesting because it has
good mathematical and computational properties and because it contains
many special numbers that are of independent interest. We conclude
this paper with some possible directions of future research.

\smallskip\goodbreak\noindent
{\bf Evaluation at singularities.} 
While every singularity of a D-finite function must also be a
singularity of its annihilating operator, the converse is in general
not true.  We have seen above that evaluating a D-finite function at a
point which is not a singularity of its annihilating operator yields a
D-finite number.  It would be natural to wonder about the values of a
D-finite function at singularities of its annihilating operator,
including those at which the given function is not analytic but its
evaluation is finite.  Also, we always consider zero as an ordinary
point of the annihilating operator.  If this is not the case, the
method used in the paper fails, as pointed out by part~\ref{EX:bessel}
of Example~\ref{EX:inside}.

\smallskip\goodbreak\noindent
{\bf Quotients of D-finite numbers.}
The set of algebraic numbers forms a field, but we do not have a
similar result for D-finite numbers. It is known that the set of
D-finite functions does not form a field. Instead, Harris and
Sibuya~\cite{HaSi1985} showed that a D-finite function~$f$ admits a
D-finite multiplicative inverse if and only if $f'/f$ is algebraic.
This explains for example why both $\e$ and $1/\e$ are D-finite, but
it does not explain why both $\pi$ and $1/\pi$ are D-finite. It would
be interesting to know more precisely under which circumstances the
multiplicative inverse of a D-finite number is D-finite. Is
$1/\log(2)$ a D-finite number? Are there choices of $R$ and $\set F$
for which $\cal D_{R,\set F}$ is a field?

\smallskip\goodbreak\noindent
{\bf Roots of D-finite functions.}
A similar pending analogy concerns compositional inverses. We know
that if $f$ is an algebraic function, then so is its compositional
inverse~$f^{-1}$. The analogous statement for D-finite functions is
not true. Nevertheless, it could still be true that the values of
compositional inverses of D-finite functions are D-finite numbers,
although this seems somewhat unlikely. A particularly interesting
special case is the question whether (or under which circumstances)
the roots of a D-finite function are D-finite numbers.

\smallskip\goodbreak\noindent
{\bf Evaluation at D-finite number arguments.}
We see that the class $\cal C_{\set F}$ of limits of convergent
C-finite sequences is the same as the values of rational functions at
points in~$\set F$, namely the field~$\set F$.  Similarly, the class
$\cal A_{\set F}$ of limits of convergent algebraic sequences
essentially consists of the values of algebraic functions at points
in~$\bar{\set F}$.  Continuing this pattern, is the value of a
D-finite function at a D-finite number again a D-finite number? If so,
this would imply that also numbers like $\e^{\e^{\e^{\e}}}$ are
D-finite. Since $1/(1-z)$ is a D-finite function, it would also imply
that D-finite numbers form a field.

\section*{Acknowledgments}

We would like to thank Alin Bostan, Marc Mezzarobba, Stephane Fischler
and the anonymous referee for their helpful comments.

\end{document}